\documentclass[a4paper, 11pt, titlepage, fleqn]{article}
\usepackage{times}
\usepackage{amsthm}
\newtheorem{thm}{Theorem}
\newtheorem{cor}{Corollary}

\newtheorem{lemma}[thm]{Lemma}
\newtheorem{prop}{Proposition}

\newcommand{\ie}{i.e.~}

\usepackage{amssymb,amsmath}

\DeclareMathOperator{\F}{\mathbb{F}}

\DeclareMathOperator{\Tr}{Tr}

\begin{document}
		\baselineskip=16.3pt
		\parskip=14pt
		\begin{center}
			\section*{The Number of Irreducible Polynomials with the First Two Coefficients Fixed over Finite Fields  of Odd Characteristic}
			
			{\large 
				Gary McGuire\footnote{Research supported by Science Foundation Ireland Grant 13/IA/1914} and
				Emrah Sercan Y{\i}lmaz \footnote {Research supported by Science Foundation Ireland Grant 13/IA/1914} 
				\\
				School of Mathematics and Statistics\\
				University College Dublin\\
				Ireland}
		\end{center}
		\subsection*{Abstract}
In this paper we give a different proof of Kuz'min's result on the number of irreducible polynomials  with the first two coefficients fixed. Our technique is to relate the question to the number of points on a curve, and to calculate the L-polynomial of the curve.

\section{Introduction}\label{sec:intro}

For $p$ a prime and $r \ge 1$ let $\mathbb F_q$ denote the finite field of $q = p^r$ elements. We count the number of monic irreducible polynomials over $\F_q$ of degree $n$ for which the first two coefficients 
have the prescribed values $t_1,t_2$, while the remaining coefficients are arbitrary, \ie 
considering irreducible polynomials of the form
\[
x^n + t_1 x^{n-1}  + t_2 x^{n-2}+ \cdots + a_{n-1}x + a_n,
\]
for fixed $t_1$, $t_2$. Denote the number of such polynomials by $I_q(n,t_1,t_2)$. 

For $a \in \mathbb F_{q^n}$, the characteristic polynomial of $a$ with respect to the extension $\mathbb F_{q^n}/\mathbb F_q$ is defined to be: 
\begin{equation}\label{characteristicpoly}
	\prod_{i = 0}^{n-1} (x - a^{q^i}) = x^n - a_{n-1}x^{n-1} + \cdots +(-1)^{n-1}a_1x +(-1)^{n} a_0.
\end{equation}
The coefficient $a_{n-1}$ is known as the trace of $a$, and  $a_{n-2}$ is known as the subtrace of $a$. We are concerned with the first two traces, 
so we let $T_1(a)=a_{n-1}$ and $T_2(a)=a_{n-2}$.
By~(\ref{characteristicpoly}) they are given by the following expressions:
\begin{eqnarray}
	\nonumber T_1(a) = \sum_{i = 0}^{n-1} a^{q^i}\;\;\; \text{ and } \;\;\; T_2(a) = \sum_{0 \le i < j \le n-1} a^{q^i + q^j}.
\end{eqnarray}
For $t_1,t_2 \in \mathbb F_q$ let $F_q(n,t_1,t_2)$ be the number of elements $a \in \mathbb F_{q^n}$ for which
$T_1(a) = t_1$ and $T_2(a) = t_2$.
In the next section we will show that $F_q(n,0,0)$ determines $I_q(n,0,0)$.
In Section \ref{t1t2}) we will show that it is enough to find $I_q(n,0,0)$ in order to find $I_q(n,t_1,t_2)$ for all $t_1,t_2\in \mathbb F_q$.

The rest of this paper is  organized as follows.
In Section \ref{sectrel} we relate $F_q(n,0,0)$ to a curve $C$ defined over $\F_q$.
In Section \ref{backg} we give the background for the proof of the main result, and Section
\ref{sectzeta} contains the proof. The proof is the calculation of the
L-polynomial of $C$.

The result of this paper was first proved by Kuz'min \cite{kuzmin1} using different methods.
Reference ~\cite{three-coef} concerns the case of the first three coefficients fixed, in even characteristic.  We conclude the paper with comments about this case in odd characteristic. Further results on higher numbers of coefficients will appear in Granger \cite{higher-no-coef}.

\section{Computing $I_q(n,0,0)$ from $F_q(n,0,0)$}\label{sec:irreduciblecounts}

 In this section we express $I_q(n,0,0)$ in terms of $F_q(n,0,0)$, using simple extensions of results from~\cite{cattell} and~\cite{yucasmullen}.  Let $\mu(\cdot)$ be the M\"obius function, and for a proposition $S$ let 
$[S]$ denote its truth value, \ie $[S] = 1$ if $S$ is true and $0$ if $S$ is false, with $1$ and $0$ interpreted as integers. The 
formula is given by the following theorem.
\begin{thm}\label{thm:irreducibles}
	Let $n \ge 2$. Then
	\begin{eqnarray}
		\label{thm91} I_q(n,0,0) &=& \frac{1}{n} \sum_{\substack{d \mid n, \:p \nmid d}} \mu(d) \big( F_q(n/d,0,0) - [p \text{ divides } n]q^{n/pd} \big)
	\end{eqnarray}
\end{thm}

For $\beta \in \mathbb F_{q^n}$ let $P = \text{Min}(\beta)$ denote the minimum polynomial of $\beta$ over $\mathbb F_q$, which has degree $n/d$ for some $d \mid n$.
Note that $T_i(\beta)$ is the coefficient of $x^{n-i}$ in $P^d$~\cite[Lemma 2]{cattell}, and so abusing notation we also write $T_i(\beta)$ as 
$T_i(P^d)$.
\begin{lemma}
	For each integer $d \ge 1$ and $P(x) \in \mathbb F_q[x]$,
	\begin{enumerate}
		\item $T_1(P^d) = d T_1(P)$
		\item $T_2(P^d) = \binom{d}{2} T_1(P) + d T_2(P)$
	\end{enumerate}
\end{lemma}

\begin{prop} Suppose that $F$ and $f$ are functions on numbers. Then
	$$F(n)=\sum\limits_{d\mid n, \; p \nmid d	}f(n/d) \;\;\;\text{ if and only if } \;\;\; f(n)=\sum\limits_{d\mid n, \; p \nmid d	}\mu(n)F(n/d).$$
\end{prop}

\begin{proof}(\textit{Proof of Theorem \ref{thm:irreducibles}}) We have 
	\begin{align*}
		\hspace{-0.8cm}F_q(n,0,0)&=\left|\bigcup_{\beta \in \mathbb F_{q^n}, \: \Tr_1(\beta)=0, \Tr_2(\beta)=0} \text{Min}(\beta)\right|\\
		&=\left|\bigcup_{d\mid n}\frac nd \left\{P\in \text{Irr}\left(\frac nd\right) : d\Tr_1(P)=0,\: \binom d2\Tr_1(P)+d\Tr_2(P)=0 \right\}\right|\\
		&=[p \text{ divides } n]\left|\bigcup_{d\mid n, \; p \mid d}\frac nd \left\{P\in \text{Irr}\left(\frac nd\right) \right\}\right| \\
		&\qquad +\left|\bigcup_{d\mid n, \: p \nmid d}\frac nd \left\{P\in \text{Irr}\left(\frac nd\right) : \Tr_1(P)=0,\: \Tr_2(P)=0 \right\}\right|\\
		&=[p \text{ divides } n]\sum\limits_{d\mid n, \: p \mid d}\frac nd I_q\left(\frac nd\right)
		+\sum\limits_{d\mid n, \: p\nmid d}\frac nd I_q\left(\frac nd,0,0\right)\\&=[p \text{ divides } n]q^{n/p}+\sum\limits_{d\mid n, \: p\nmid d}\frac nd I_q\left(\frac nd,0,0\right).
	\end{align*}	 Therefore, $$I_q(n,0,0)=\frac1n\sum\limits_{d\mid n, \: p\nmid d}\left(F(n/d,0,0)-[p \text{ divides } n] q^{n/pd}\right).$$
\end{proof}

\section{Relation between $F_q(n,0,0)$ and the curve $y^q-y=x^{q+1}-x^2$}\label{sectrel}

In this section we will prove Lemma \ref{lem-T2a+b} and Lemma \ref{lem-F+curve} in order to obtain the relation  between $F_q(n,0,0)$ and the curve $y^q-y=x^{q+1}-x^2$. This relation is stated in Proposition \ref{prop-relation}.

\begin{lemma}\label{lem-T2a+b}
	For all $\alpha, \beta \in \mathbb F_{q^n}$, we have
	$$T_2(\alpha+\beta)=T_2(\alpha)+T_2(\beta)+T_1(\alpha)T_1(\beta)-T_1(\alpha\beta).$$
\end{lemma}
\begin{proof} We have
	\begin{align*}
	T_2(\alpha+\beta) &= \sum_{0 \le i < j \le n-1} (\alpha+\beta)^{q^i + q^j}\\&=\sum_{0 \le i < j \le n-1}\alpha^{q^i+q^j}+\sum_{0 \le i < j \le n-1}\beta^{q^i+q^j}+\sum_{0 \le i < j \le n-1}(\alpha^{q^i}\beta^{q^j}+\beta^{q^i}\alpha^{q^j})\\&=T_2(\alpha)+T_2(\beta)+\sum_{0 \le i\le n-1}\alpha^{q^i}\sum_{0 \le j \le n-1}\beta^{q^j}-\sum_{0 \le k \le n-1}(\alpha\beta)^{q^k}\\&=T_2(\alpha)+T_2(\beta)+T_1(\alpha)T_1(\beta)-T_1(\alpha\beta).
	\end{align*}
\end{proof}
\begin{lemma} \label{lem-F+curve}
	For all $c \in \mathbb F_{q^n}$ we have $T_2(c^q-c)=T_1(c^{q+1}-c^2)$.
\end{lemma}
\begin{proof}
	Let $\alpha=c^q$ and $\beta=-c$ in Lemma \ref{lem-T2a+b}. Then we have $T_2(c^q-c)=2T_2(c)-T_1(c)^2+T_1(c^{q+1})$. Moreover, we have $$2T_2(c)=2\sum_{0 \le i < j \le n-1}c^{q^i+q^j}=\sum_{0 \le i \le n-1}c^{q^i}\cdot\sum_{0 \le j \le n-1}c^{q^j}-\sum_{0 \le k \le n-1}(c\cdot c)^{q^k}=T_1(c)^2-T(c^2).$$ Therefore, $T_2(c^q-c)=-T_1(c^2)+T_1(c^{q+1})=T_1(c^{q+1}-c^2)$.
\end{proof}

Here is the important relation to the curve.

\begin{prop}\label{prop-relation}
	We have $F_q(n,0,0)=\frac1q \left\{ x \in \mathbb F_{q^n} : T_1(x^{q+1}-x^2)=0\right\}$.
\end{prop}
\begin{proof}
	Let $y\in \mathbb F_{q^n}$. Then $T_1(y)=0$ if and only if there exists  $x \in \mathbb F_{q^n}$ such that $y=x^q-x$. Moreover, by Lemma \ref{lem-F+curve}, we have $T_2(y)=T_2(x^q-x)=T_1(x^{q+1}-x^2)$. Since the map $x \to x^q-x$ is a $q$-to-$1$ map, we have $F_q(n,0,0)=\frac1q \left\{ x \in \mathbb F_{q^n} : T_1(x^{q+1}-x^2)=0\right\}$.
\end{proof}

Up to here, the results are valid in odd or even characteristic. We have done the even characteristic case  in \cite{three-coef}. For the rest of this paper, $q$ will be odd.

\section{Background}\label{backg}

\subsection{Quadratic forms}\label{QF}

We now recall the basic theory of quadratic forms over $\mathbb{F}_{q}$, where $q$ is odd.

Let $K=\mathbb{F}_{q^n}$, and 
let $Q:K\longrightarrow \mathbb{F}_{q}$ be a quadratic form.
The polarization of $Q$ is the symplectic bilinear form $B$ defined by $B(x,y)=Q(x+y)-Q(x)-Q(y)$. By definition the radical of $B$ (denoted $W$) is $
W =\{ x\in K : B(x,y)=0 \text{  for all $y\in K$}\}$. The rank of $B$ is defined to be $n-\dim(W)$. The rank of $Q$ is defined to be the rank of $B$.

The following result is well known, see Chapter $6$ of \cite{lidl}  for example.
\bigskip

\begin{prop}\label{counts}
	Continue the above notation. Let $N=|\{x\in K : Q(x)=0\}|$, and let $w=\dim(W)$. If $Q$ has odd rank then $N=q^{n-1}$; if $Q$ has even rank then $N=q^{n-1}\pm (q-1)q^{(n-2+w)/2}$.
\end{prop}

The following consequence is important for this paper.

\begin{prop} \label{Szero}
Let $Q(x)=Tr(x^{q+1}-x^2)$, and let $N=|\{x\in K : Q(x)=0\}|$. Then $N=q^{n-1}$ only if $n$ is an even integer with $(n,p)=1$, or an odd integer with $(n,p)=p$.	
\end{prop}
\begin{proof}
 We have $B(x,y)=Q(x+y)-Q(x)-Q(y)=T_1(x^qy+y^qx-2xy)=T_1(y^q(x^{q^2}-2x^q+x))$ and hence $W =\{ x\in K : B(x,y)=0 \text{  for all $y\in K$}\}=\{x\in K : x^{q^2}-2x^q+x=0\}$. 
 Let $l(x)=x^2-2x+1=(x-1)^2$. Then 
 $$
 (l(x),x^n-1)=\begin{cases}
 \;x-1 &\text{ if $(n,p)=1$,}\\
 (x-1)^2 &\text{ if $(n,p)=p$}
 \end{cases}
 $$
 and so 
 $$
 w=\begin{cases}
 1 &\text{ if $(n,p)=1$,}\\
 2 &\text{ if $(n,p)=p$.}
 \end{cases}
 $$ 
 Therefore $N=q^{n-1}$ only if $n$ is an even integer with $(n,p)=1$ or an odd integer with $(n,p)=p$ by  Proposition \ref{counts}, since $n-w$ is odd only in these situations.
 	
\end{proof}

\subsection{Maximal and minimal curves}

	Let $q=p^r$ where $p$ is any prime and $r\ge 1$ is an integer. Let $X$ be a projective smooth absolutely irreducible curve of genus $g$ defined over $\mathbb{F}_q$.
Consider the $L$-polynomial of the curve $X$ over $\mathbb F_{q}$, defined by
$$L_{X/\mathbb{F}_q}(T)=L_X(T)=exp\left( \sum_{n=1}^\infty ( \#X(\mathbb F_{q^n}) - q^n - 1 )\frac{T^n}{n}  \right).$$
where $\#X(\mathbb F_{q^n})$ denotes the number of $\mathbb F_{q^n}$-rational points of $X$. 
It is well known that $L_X(T)$ is a polynomial of degree $2g$ with integer coefficients, so we write it as 
\begin{equation} \label{L-poly}
L_X(T)= \sum_{i=0}^{2g} c_i T^i, \ c_i \in \mathbb Z.
%L_X(T)=c_0+c_1T+c_2T^2+\cdots + c_{2g-1}T^{2g-1}+c_{2g}T^{2g}.
\end{equation}
It is also well known that $c_0=1$ and $c_{2g}=q^g$.

Let $\eta_1,\cdots,\eta_{2g}$ be the roots of the reciprocal of the $L$-polynomial of $X$ over 
$\mathbb F_{q}$ (sometimes called the Weil numbers of $X$). Then, for any $n\geq 1$, 
the number of rational points of $X$ over $\mathbb F_{q^{n}}$ is
given by
\begin{equation}\label{eqn-sum of roots}
 \#X(\mathbb F_{q^{n}})=(q^{n}+1)- \sum\limits_{i=1}^{2g}\eta_i^n. 
\end{equation}
The Riemann Hypothesis for curves over finite fields  
states that $|\eta_i|=\sqrt{q}$ for all $i=1,\ldots,2g$. 
It follows immediately that
\begin{equation}
|\#X(\mathbb F_{q^n})-(q^n+1)|\leq 2g\sqrt{q^n}
\end{equation}
which is the Hasse-Weil bound.

We call $X(\mathbb F_{q})$ \emph{maximal} if $\eta_i=-\sqrt{q}$ for all $i=1,\cdots,2g$, so the Hasse-Weil upper bound is met. Equivalently, $X(\mathbb F_{q})$ is maximal
if and only if $L_X(T)=(1+\sqrt{q} T)^{2g}$.

We call $X(\mathbb F_{q})$ \emph{minimal} if $\eta_i=\sqrt{q}$ for all $i=1,\cdots,2g$.
Equivalently, $X(\mathbb F_{q})$ is minimal
if and only if $L_X(T)=(1-\sqrt{q} T)^{2g}$.

Note that if $X(\mathbb F_{q})$ is minimal or maximal then $q$ is a square (i.e.\ $r$ is even).

The following properties follow immediately.

\begin{prop} \label{minimal-prop}
	\begin{enumerate}
	\item If $X(\mathbb F_{q})$ is maximal then  $X(\mathbb F_{q^{n}})$ is minimal for even $n$ and maximal for odd $n$. 
	\item  If $X(\mathbb F_{q})$ is minimal then  $X(\mathbb F_{q^{n}})$ is minimal for all $n$.
	\end{enumerate}
\end{prop}
%\begin{proof}
%	Since $\eta_i^n=\pm \sqrt{q^{n}}$ for all $i=1,\cdots,2g$, we have  $\eta_i^{2n}=(\pm \sqrt{q^{n}})^2=\sqrt{q^{2n}}$ for all $i=1,\cdots,2g$. The proof follows by definition of a minimal curve.
%\end{proof}

\subsection{Supersingular Curves}\label{supsing}

A curve $X$ of genus $g$ defined over $\mathbb F_q$ is \emph{supersingular} if any of the following 
equivalent properties hold.

\begin{enumerate}
\item All Weil numbers of $X$ have the form $\eta_i = \sqrt{q}\cdot \zeta_i$ where $\zeta_i$ is a root of unity.
\item The Newton polygon of $X$ is a straight line of slope $1/2$.
\item The Jacobian of $X$ is geometrically isogenous to $E^g$ where
$E$ is a supersingular elliptic curve.
\item If $X$ has $L$-polynomial
$L_X(T)=1+\sum\limits_{i=1}^{2g} c_iT^i$ 
then
$$ord_p(c_i)\geq \frac{ir}{2}, \ \mbox{for all $i=1,\ldots
,2g$,}$$ where $q=p^r$.
\end{enumerate}

By the first property, a supersingular curve defined over $\mathbb F_q$ becomes minimal over some finite extension of $\mathbb F_q$.
Conversely, any minimal or maximal curve is supersingular.

\section{Calculation of L-Polynomial of  $C:y^q-y=x^{q+1}-x^2$}\label{sectzeta}

First note that by Proposition \ref{prop-relation}, in order to complete
the results of this paper
we need to find the number $N$ of $x\in \mathbb{F}_{q^n}$ with $T_1(x^{q+1}-x^2)=0$.
Because elements of trace 0 have the form $y^q-y$,  finding $N$ is equivalent
to finding the exact number of $\mathbb{F}_{q^n}$-rational points on $C$.
Indeed, 
\begin{equation}\label{quadpts}
\#C(\mathbb F_{q^n})=qN+1.
\end{equation}

In this section we present the calculation of the exact number of $\mathbb{F}_{q^n}$-rational
points on $C:y^q-y=x^{q+1}-x^2$, for all $n$, which is equivalent to calculating the L-polynomial of $C$.
This will complete the proof of the results in this paper.

We will need the following elementary lemmas later. 

\begin{lemma}\label{lemma-unity}
	\begin{enumerate} \item If $p\equiv 1 \pmod 4$ then $\sqrt{p}$ is an element of the splitting field of $x^{2p}-1$ over $\mathbb Q$ and not an element of the splitting field of $x^{2p}+1$ over $\mathbb Q$. 
	\item If $p\equiv 3 \pmod 4$ then $\sqrt{p}$ is an element of the splitting field of $x^{2p}+1$ over $\mathbb Q$ and not an element of the splitting field of $x^{2p}-1$ over $\mathbb Q$.
	\end{enumerate}
\end{lemma}

Lemma \ref{lemma-unity} is Exercise 2.1 in ~\cite{cyclotomic}.

\begin{lemma}\label{lemma-unity2}
	$1$ cannot be written as a $\mathbb Q$-linear combination of the roots of $x^{2p}+1$.
\end{lemma}
\begin{proof} 
	Let $w$ be a primitive complex $p$-th root of unity, and let $i$ 
	be a primitive complex $4$-th root of unity. 
	Then the set of roots of $x^{2p}+1$ is $\{\pm i w^k\: : \: 0 \le k \le p-1\}$. If $1$ is  a linear combination of the roots of $x^{2p}+1$ over $\mathbb Q$, then there exist $a_1,\cdots,a_n \in \mathbb Q$ such that $$i\sum_{k=0}^{p-1}a_kw^k=1.$$ 
	But then $i\in \mathbb  Q(w)$, which is a contradiction to
	$\mathbb Q(w) \cap \mathbb  Q(i)=\mathbb Q$ (see Proposition 2.4 in \cite{cyclotomic}). 
\end{proof}

\begin{lemma}\label{max-min}
	If $q\equiv 1 \pmod 4$ then $C(\mathbb F_{q^{2p}})$ is minimal and if $q\equiv 3 \pmod 4$ then $C(\mathbb F_{q^{2p}})$ is maximal and $C(\mathbb F_{q^{4p}})$ is minimal.
\end{lemma}

\begin{proof}  Note that the genus of the curve $C$ is $q(q-1)/2$.

We consider $C$ over the extension field $\mathbb F_{q^{n}}$ where $n=2p$. By Proposition \ref{Szero} and its proof, $N=q^{n-1}\pm (q-1)q^{n/2}$ and so 
$\#C(\mathbb F_{q^{n}}) - (q^n+1)=\pm q(q-1)q^{n/2}$ by equation \eqref{quadpts}.
Therefore  $C(\mathbb F_{q^{2p}})$ is maximal or minimal. 	

By Proposition  \ref{minimal-prop},  $C(\mathbb F_{q^{4p}})$ is minimal, for any $q$.
Therefore, by equation \eqref{eqn-sum of roots}, for any $n$ we have 
	 \begin{equation}\label{eqn-sum of roots2}
	\#C(\mathbb F_{q^n})-(q^n+1)=-\sum_{j=0}^{q(q-1)} \eta_i^n
	\end{equation}  
	where $\eta_j=\sqrt{q}\zeta_j$ and the $\zeta_j$ are $(4p)$-th roots of unity.
	We  rewrite \eqref{eqn-sum of roots2} as 
	\begin{equation}\label{eqn-roots-of-unity}
	-q^{-n/2}\left[\#C(\mathbb F_{q^n})-(q^n+1)\right]=\sum_{i=0}^{q(q-1)} \zeta_i^n.
	\end{equation}  
		Considering $C$ over $\F_q$ now, since  $y^q-y=x^{q+1}-x^2=0$ for all $x,y\in \mathbb F_q$, we have 
	\begin{equation*}
	\#C(\mathbb F_{q})=q^2+1=q+1+q(q-1)
	\end{equation*} 
	and therefore 
	\begin{equation}\label{eqn-n=1}
	-q^{-1/2}\left[\#C(\mathbb F_q)-(q+1)\right]=-\sqrt{q}(q-1).
	\end{equation} 
	Putting $n=1$ in \eqref{eqn-roots-of-unity} and using \eqref{eqn-n=1}  
	we  obtain
	\begin{equation}\label{sumn1}
	\sum_{i=0}^{q(q-1)} \zeta_i= -\sqrt{q}(q-1).
	\end{equation}
The $\zeta_j$ are roots of $x^{4p}-1=(x^{2p}-1)(x^{2p}+1)$. Furthermore,
	all $\zeta_j$ are roots of $x^{2p}-1$ if and only if $C(\mathbb F_{q^{2p}})$ is minimal,
	and all $\zeta_j$ are roots of $x^{2p}+1$ if and only if $C(\mathbb F_{q^{2p}})$ is
	maximal.
	
	Recall from the start of this proof that $C(\mathbb F_{q^{2p}})$ is either maximal or minimal. 
	
	{\bf Case 1.} Suppose $q$ is a square (which implies $q\equiv 1 \pmod 4$). 
	If $C(\mathbb F_{q^{2p}})$ is
	maximal then \eqref{sumn1} expresses $1$ as a $\mathbb{Q}$-linear combination of 
the roots of $x^{2p}+1$, which contradicts Lemma \ref{lemma-unity2}.
Therefore $C(\mathbb F_{q^{2p}})$ is minimal.

{\bf Case 2A.} Suppose $q$ is a nonsquare and that $q\equiv 1 \pmod 4$. 
	If $C(\mathbb F_{q^{2p}})$ is
	maximal then \eqref{sumn1} expresses $\sqrt{p}$ as a $\mathbb{Q}$-linear combination of 
the roots of $x^{2p}+1$, which contradicts Lemma \ref{lemma-unity} part 1.
Therefore $C(\mathbb F_{q^{2p}})$ is minimal.

{\bf Case 2B.} Suppose $q$ is a nonsquare and that $q\equiv 3 \pmod 4$. 
	If $C(\mathbb F_{q^{2p}})$ is
	minimal then \eqref{sumn1} expresses $\sqrt{p}$ as a $\mathbb{Q}$-linear combination of 
the roots of $x^{2p}-1$, which contradicts Lemma \ref{lemma-unity} part 2.
Therefore $C(\mathbb F_{q^{2p}})$ is maximal.
\end{proof}

\begin{cor}
$C$ is a supersingular curve.
\end{cor}

The corollary follows from the proof of the Lemma.

%\subsection{Finding the Number of Rational Points of the Curve $C:y^q-y=x^{q+1}-x^2$ over $\mathbb F_q$}

Finally, we calculate the number of $\mathbb F_{q^n}$rational points of the curve $C:y^q-y=x^{q+1}-x^2$.

\begin{thm} \label{mainL}
Let $q=p^r$. Then we have
	$$\#C(\mathbb F_{q^n})-(q^n+1)=\begin{cases}
	0 &\text{if $(n,2p)=2$ or $p$},\\
	(-1)^{(n-1)/2}(q-1)q^{\frac n2+1}  &\text{if $(n,2p)=2p$},\\
	\left(\dfrac{-n}{p}\right)^r (q-1) q^{(n+1)/2} & \text{if $(n,2p)=1$}
	\end{cases}$$
	where $\left(\dfrac{-n}{p}\right)$ is the Legendre symbol.
\end{thm}

\begin{proof}
	The cases $(n,2p)=2$ or $p$ come from Proposition \ref{Szero} (they are the cases
	$N=q^{n-1}$).
	The  case $(n,2p)=2p$ follows from Lemma \ref{max-min} and Proposition \ref{minimal-prop}.
	In the rest of proof we will assume $(n,2p)=1$.
	
As a continuation of the proof of Lemma \ref{max-min}, let $w=e^{\frac{2\pi i}{p}}$. Then we have \begin{equation}\label{eqn-in-roots}
	\zeta_j \in \begin{cases}
	\{\pm w^k \: : \: 0\le k \le p-1\} &\text{ if }  q\equiv 1 \pmod 4, \\
	\{\pm i w^k\: : \: 0 \le k \le p-1\} &\text{ if } q\equiv 3 \pmod 4.\\ \end{cases}  
	\end{equation} 
	
	Recall equation \eqref{sumn1}
	\[
	\sum_{i=0}^{q(q-1)} \zeta_i= -\sqrt{q}(q-1).
	\]
 By this and equation \eqref{eqn-in-roots},
 there exist non-negative integers $b_0,c_0,\cdots, b_{p-1},c_{p-1}$ 
 (the multiplicities of the $\zeta_j$) such that 
 \begin{equation}\label{eqn-general-form-sum}
	-\sqrt{q}(q-1)=\begin{cases}
	b_0\cdot 1+ c_0\cdot(-1)+\displaystyle\sum_{k=1}^{p-1}b_kw^k+\sum_{k=1}^{p-1}c_k(-w^k) &\text{ if } q\equiv 1 \mod 4,\\
	b_0\cdot i+ c_0\cdot(-i)+\displaystyle\sum_{k=1}^{p-1}b_k(iw^k)+\sum_{k=1}^{p-1}c_k(-iw^k)  &\text{ if } q\equiv 3 \mod 4. \end{cases}
	\end{equation} 
	Define $a_i=b_i-c_i$ for $0 \le i \le p-1$, then $a_1,\cdots,a_{p-1}$ are in $\mathbb Z$ and we have 
	\begin{equation} \label{eqn-sum-with -0}
	\sum_{k=0}^{p-1}a_kw^k=\begin{cases}
	-\sqrt{q}(q-1) &\text{ if } q\equiv 1 \mod 4,\\
	i\sqrt{q}(q-1)&\text{ if } q\equiv 3 \mod 4. \end{cases}
	\end{equation} Since the sum of roots of $x^p-1$ is $0$, we have \begin{equation} \label{eq-sum-minus1}
	\sum_{k=1}^{p-1}w^k=-1.
	\end{equation} Therefore we can rewrite equation \eqref{eqn-sum-with -0} as \begin{equation*}
	\sum_{k=1}^{p-1}(a_k-a_0)w^k=\begin{cases}
	-\sqrt{q}(q-1) &\text{ if } q\equiv 1 \pmod 4,\\
	i\sqrt{q}(q-1)&\text{ if } q\equiv 3 \pmod 4. \end{cases}
	\end{equation*} 
	
{\bf Case 1.}	Suppose $q$ is a square (which implies $q\equiv 1 \pmod 4$). 
	We write the previous equation as 
	$$\sum_{k=1}^{p-1}(a_k-a_0)w^k=-\sqrt{q}(q-1)=\sqrt{q}(q-1)\sum_{k=1}^{p-1}w^k=\sum_{k=1}^{p-1}\sqrt{q}(q-1)w^k.$$ 
	Since the set $W=\{w,w^2,\cdots, w^{p-1}\}$ is the set of  roots of the $p$-th cyclotomic polynomial, which is irreducible over $\mathbb Q$, $W$ is a linearly independent set over $\mathbb Q$. Therefore, 
	 $a_k-a_0=\sqrt{q}(q-1)$ for all $1\le k \le p-1$. In particular, we have $a_1=a_2=\cdots=a_{p-1}$.
	
	Since $(n,2p)=1$, the map $\sigma: x\mapsto x^n$ permutes the set $\{\pm1,\pm w,\cdots,\pm w^{p-1}\}$ with the following properties: \begin{itemize}
		\item $\sigma$ fixes $1$ and $-1$,
		\item $\sigma$ permutes $\{ w,\cdots, w^{p-1}\}$ and $\{ -w,\cdots, -w^{p-1}\}$,
		\item if $\sigma(x)= y$ then $\sigma(-x) = -y$.
	\end{itemize} 
	Therefore  \begin{align*}
	-q^{-n/2}\left[\#C(\mathbb F_{q^n})-(q^n+1)\right]&=\sum_{i=0}^{q(q-1)} \zeta_i^n \quad \textrm{by } \eqref{eqn-roots-of-unity}\\
	&= b_0\cdot 1^n+ c_0\cdot(-1)^n+\displaystyle\sum_{k=1}^{p-1}b_k(w^k)^n+\sum_{k=1}^{p-1}c_k(-w^k)^n\\ 
	&=\sum_{k=0}^{p-1}a_kw^k  \quad \textrm{by the properties of $\sigma$} \\ 
	&=-\sqrt{q}(q-1) \quad \textrm{by } \eqref{eqn-sum-with -0}.
	\end{align*}  
	This gives 
	$$\#C(\mathbb F_{q^n})-(q^n+1)=q^{(n+1)/2}(q-1)$$
	which completes the proof in the case that $q$ is a square.
	
	{\bf Case 2.} Now assume that $q$ is non-square. 
	Using equation \eqref{eq-sum-minus1} we write the expression in \eqref{eqn-general-form-sum} as
	$$b_0\cdot 1+ c_0\cdot(-1)+\displaystyle\sum_{k=1}^{p-1}b_kw^k+\sum_{k=1}^{p-1}c_k(-w^k)=\sum_{k=1}^{p-1}(b_k-c_k-b_0+c_0)w^k.$$
	Let $a_k := b_k-c_k-b_0+c_0$ for $1\le k \le p-1$.
	
	Gauss proved that (see Chapter 6 in \cite{ireland}) 
	\begin{equation}\label{eqn-indep-sqrtq}
	\sum_{k =1}^{p-1}\left(\frac k p\right)w^{k}=\begin{cases}
	\sqrt{p} & \text{ if } p\equiv 1 \mod 4,  \\
	i\sqrt{p} & \text{ if } p\equiv 3 \mod 4.
	\end{cases} 
	\end{equation} 
By equations \eqref{eqn-general-form-sum} and \eqref{eqn-indep-sqrtq}  we have that 
\begin{equation}\label{eqn-1-3-mod4}
\sum_{k=1}^{p-1}a_kw^k=\begin{cases}
-\sqrt{q}(q-1)=\displaystyle\sum_{k =1}^{p-1}-\sqrt{q/p}\ (q-1)\left(\frac k p\right)w^{k} &\text{ if } q\equiv 1 \pmod 4,\\[20pt]
\;i\sqrt{q}(q-1)=\displaystyle\sum_{k =1}^{p-1}\sqrt{q/p}\ (q-1)\left(\frac k p\right)w^{k}&\text{ if } q\equiv 3 \pmod 4. \end{cases}
\end{equation}  
By the uniqueness of $\mathbb Q$-linear combinations of a linearly independent set  over 
$\mathbb Q$, by \eqref{eqn-1-3-mod4} we get
\begin{equation}\label{eqn-coefficients}
a_k =\begin{cases}
-\sqrt{q/p}(q-1)\left(\frac k p\right) &\text{ if } q\equiv 1 \mod 4,\\[20pt]
\sqrt{q/p}(q-1)\left(\frac k p\right)&\text{ if } q\equiv 3 \mod 4 \end{cases}
	\end{equation} for all $1\le k \le p-1$.

	Note that the equality  $\left(\frac{kn}{p}\right)=\left(\frac{k}{p}\right)\left(\frac{n}{p}\right)$ holds for all integers $k,n\ge 1$ which are relatively prime to $p$. 
	
	{\bf Case 2A.} If $q\equiv 1 \mod 4$ we have 
	\begin{align*}
	&-q^{-n/2}\left[\#C(\mathbb F_{q^n})  -(q^n+1)\right] \\
	&=\sum_{i=0}^{q(q-1)} \zeta_i^n\\
	&= b_0\cdot 1^n+ c_0\cdot(-1)^n+\displaystyle\sum_{k=1}^{p-1}b_k(w^k)^n+\sum_{k=1}^{p-1}c_k(-w^k)^n\\
	&= (-b_0+c_0)\displaystyle\sum_{k=1}^{p-1}w^{kn}+\displaystyle\sum_{k=1}^{p-1}b_k(w^k)^n+\sum_{k=1}^{p-1}c_k(-w^k)^n \\ 
	&=\sum_{k=1}^{p-1}a_kw^{kn} \\
	&=\sum_{k=1}^{p-1}\left(\frac{n}{p}\right)a_kw^{k}&&\text{by \eqref{eqn-coefficients}}\\
	&=-\left(\frac{n}{p}\right)\sqrt{q}(q-1). &&\text{by \eqref{eqn-1-3-mod4}}
	\end{align*}
	
	{\bf Case 2B.} If $q\equiv 3 \mod 4$ we have\begin{align*}\hspace{-1cm}
	&-q^{-n/2}\left[\#C(\mathbb F_{q^n})-(q^n+1)\right]\\
	&=\sum_{i=0}^{q(q-1)} \zeta_i^n\\
	&= i^n\left(b_0\cdot 1^n+ c_0\cdot(-1)^n+\displaystyle\sum_{k=1}^{p-1}b_k(w^k)^n+\sum_{k=1}^{p-1}c_k(-w^k)^n\right) \\
	&= i^n\left((-b_0+c_0)\displaystyle\sum_{k=1}^{p-1}w^{kn}+\displaystyle\sum_{k=1}^{p-1}b_k(w^k)^n+\sum_{k=1}^{p-1}c_k(-w^k)^n\right) \\ 
	&=i^n\sum_{k=1}^{p-1}a_kw^{kn} \\
	&=i^n\sum_{k=1}^{p-1}\left(\frac{n}{p}\right)a_kw^{k} &&\text{by  \eqref{eqn-coefficients}}\\&=i^{n+1}\left(\frac{n}{p}\right)\sqrt{q}(q-1). &&\text{by  \eqref{eqn-1-3-mod4}}
	\end{align*}
	
	In conclusion, we have 
	\begin{align*}
	\#C(\mathbb F_{q^n})-(q^n+1)&=-q^{rn}\sum_{k=0}^{p-1}a_kw^{kn}\\
	&=\begin{cases}
	(q-1)q^{\frac{n+1}2} & \text{ if } q \text{ is square}\\
	\left(\frac n p \right)(q-1)q^{\frac{n+1}2} & \text{ if } q \text{ is not square and } q\equiv 1 \pmod 4\\
	\left(\frac n p \right)(q-1)q^{\frac{n+1}2} & \text{ if } q\equiv 3 \pmod 4 \text{ and } n\equiv 1 \pmod 4\\
	-\left(\frac n p \right)(q-1)q^{\frac{n+1}2} & \text{ if } q\equiv 3 \pmod 4 \text{ and } n\equiv 3 \pmod 4\\
	\end{cases}\\
	\end{align*} 
\end{proof}

We have enough  information to calculate the L-polynomial explicitly, which we will do
by giving each root and its multiplicity.

\begin{cor}
Let $w$ be a primitive complex $p$-th root of unity, and let $i$ 
	be a primitive complex $4$-th root of unity. 
When $q$ is a square, the roots of the L-polynomial of $C$ (after division by $\sqrt q$) are 
\begin{itemize}
	\item $1$ with multiplicity $\left(\dfrac qp -\sqrt{q}\dfrac{p}{p-1}\right)\left( \dfrac{q-1}{2} \right)$,
	\item $-1$ with multiplicity $\left(\dfrac qp +\sqrt{q}\dfrac{p}{p-1}\right)\left( \dfrac{q-1}{2} \right)$,
	\item $w^k$ with multiplicity $\left(\dfrac qp +\dfrac{\sqrt q}{p}\right)\left( \dfrac{q-1}{2}\right)$,
	\item $-w^k$ with multiplicity $\left(\dfrac qp -\dfrac{\sqrt q}{p}\right)\left( \dfrac{q-1}{2}\right)$
\end{itemize}
for $1\le k \le p-1$.

When $q\equiv 1 \mod 4$ and  non-square, the  roots of the L-polynomial of $C$ (after division by $\sqrt q$) are 
\begin{itemize}
	\item $1$ with multiplicity $\left( \dfrac qp\right) \left(  \dfrac{q-1}{2} \right)$,
	\item $-1$ with multiplicity $\left( \dfrac qp \right) \left(  \dfrac{q-1}{2} \right)$,
	\item $w^k$ with multiplicity $\left(\dfrac qp -\left(\dfrac{k}{p}\right)\sqrt{\dfrac{ q}{p}}\right)\left(  \dfrac{q-1}{2}\right)$,
	\item $-w^k$ with multiplicity $\left(\dfrac qp +\left(\dfrac{k}{p}\right)\sqrt{\dfrac{ q}{p}}\right)\left(  \dfrac{q-1}{2}\right)$
\end{itemize}
for $1\le k \le p-1$.

When $q\equiv 3 \mod 4$, the  roots of the L-polynomial of $C$ (after division by $\sqrt q$) are 
\begin{itemize}
	\item $i$ with multiplicity $ \left( \dfrac qp \right) \left(  \dfrac{q-1}{2} \right)$,
	\item $-i$ with multiplicity $ \left( \dfrac qp \right) \left(  \dfrac{q-1}{2} \right)$,
	\item $iw^k$ with multiplicity $\left(\dfrac qp +\left(\dfrac{k}{p}\right)\sqrt{\dfrac{ q}{p}}\right) \left( \dfrac{q-1}{2}\right)$,
	\item $-iw^k$ with multiplicity $\left(\dfrac qp -\left(\dfrac{k}{p}\right)\sqrt{\dfrac{ q}{p}}\right) \left(  \dfrac{q-1}{2}\right)$
\end{itemize}
for $1\le k \le p-1$.
\end{cor}

\begin{proof}
One can find the multiplicities of the $\zeta_j$ using the proof of Theorem \ref{mainL}.
There are $2p$ linear equations for $2p$ unknowns (the multiplicities) coming from equation 
\eqref{eqn-general-form-sum} and its analogues over $\F_{q^2}$ and $\F_{q^p}$,
and the fact that the sum of the multiplicities is $2g$.
\end{proof}

\section{A Remark for $F_q(n,t_1,t_2)$ and the Final Result}\label{t1t2}

 We  find $F_q(n,t_1,t_2)$ for all $t_1,t_2 \in \mathbb F_q$ from $F_q(n,0,0)$ in the same way as ~\cite{kuzmin2}. The general case where $t_1,t_2$ are arbitrary and $p > 2$ reduces to $t_1=t_2=0$. Indeed, if $(p, n) = 1$, then the transformation $f(x) \to f(x - a~/n)$ shows that $F_q(n,t_1, t_2)= F_q(n,0, t_2-((n-1)/2n)t_1)$. If $p\mid n$ and $a_1\ne 0$, then a linear transformation of the unknown $x$ gives $F_q(n,t_1, t_2)=F_q(n,0, 1)$. Since $Q(x)=\Tr(x^{q+1}-x^2)$ is a quadratic form over $\mathbb F_q$, in order to find $|\{ x \in \mathbb F_{q^n} \: | Q(x)=t_2\}|$ for all $t_2 \in \mathbb F_q$ and so $F_q(n,0,t_2)$ for all $t_2 \in \mathbb F_q$ it is enough to find $|\{ x \in \mathbb F_{q^n} \: | Q(x)=0\}|$.

The final result for $F_q(n,0,0)$, originally proved by Kuzmin \cite{kuzmin1} by different methods,  is as follows.

\begin{thm} Let $q=p^r$. Then we have
	$$F_q(n,0,0)-q^{n-2}=\begin{cases}
	0 &\text{if $(n,2p)=2$ or $p$},\\
	(-1)^{(n-1)/2}(q-1)q^{\frac n2-1}  &\text{if $(n,2p)=2p$},\\
	\left(\dfrac{-n}{p}\right)^r (q-1) q^{(n-3)/2} & \text{if $(n,2p)=1$}
	\end{cases}$$
	where $\left(\dfrac{-n}{p}\right)$ is the Legendre symbol.
\end{thm}

\section{A Remark for Three Coefficients Fixed}

Having studied two coefficients, one may reasonably ask for the number of irreducible polynomials with the first three (or more) coefficients fixed. 
In \cite{three-coef} we (and our co-authors) found $F_q(n,0,0,0)$ for $q$ even by similar methods. Now suppose $q$ is odd. Using the same method as in ~\cite{three-coef} we get 
$$F_q(n,0,0,0)=\frac1q \left|\left\{ x \in \mathbb F_{q^n}  \: | \: \Tr_1(x^{q+2}-x^{2})=\Tr_1(x^{2q+1}-x^{q+2})=0\right\}\right|.$$In order to find this number we have to find the number of rational points of the curves $$C_1: y^q-y=x^{q+1}-x^2,$$$$C_2:y^q-y=x^{2q+1}-x^{q+2},$$$$C_3: y^q-y=x^{2q+1}-x^{q+2}+x^{q+1}-x^2$$ over $\mathbb F_{q^n}$. When $q$ is even, these curves are all supersingular (see \cite{three-coef}). 
Interestingly this does not happen for $q$ odd, as we will illustrate with an example. 

Let $q=3$. Then we have $C_1: y^3-y=x^4-x^2$, $C_2: y^3-y=x^7-x^5$ and $C_3: y^3-y=x^7-x^5+x^4-x^2$.  Using MAGMA \cite{magma} we calculated that the $L$-polynomial of $C_1$ is $$L_1(x)=(3x^2+1)(3x^2+3x+1)^2$$the $L$-polynomial of $C_2$ is $$L_2(x)=(27x^6+27x^5+27x^4+15x^3+9x^2+3x+1)^2$$ and the $L$-polynomial of $C_3$ is 
$$L_3(x)=
729x^{12}+1458x^{11}+1458x^{10}+1053x^9+576x^8+
243x^7+117x^6+81x^5+63x^4+$$   $$39x^3+18x^2+6x+1.$$ The curves $C_2$ and $C_3$ are not supersingular over $\mathbb F_3$ by the fourth property in Section \ref{supsing}.

Let $\alpha_1,\cdots,\alpha_6$ be the roots of $27x^6+27x^5+27x^4+15x^3+9x^2+3x+1$ and $\beta_1,\cdots,\beta_{12}$ be the roots of $L_3(x)$. Then we have $$F_3(n,0,0,0)=3^{n-3}+$$
$$3^{n/2-3}
\left(i^n+(-i)^n+2\left(\frac{-\sqrt3+i}{2}\right)^{n}+2\left(\frac{-\sqrt3-i}{2}\right)^{n}+2\sum_{j=1}^6 \alpha_j^{n}+(q-1)\sum_{j=1}^{12}\beta_j^n\right).$$

The appearance of the $\alpha_j$ and $\beta_j$ (which are not roots of unity) in the formula means that this formula has a fundamentally different flavour to the formula in the case of two coefficients. This explains the difficulty in extending Kuz'min's result from two to three coefficients.
These matters are studied to a much greater extent in \cite{higher-no-coef}.

\section*{Acknowledgements} 
We thank Robert Granger for  helpful conversations. 
\bibliographystyle{unsrt}		
\bibliography{supersingular}									

\begin{thebibliography}{10}

\bibitem{kuzmin1}
E.~N. Kuz'min.
\newblock A class of irreducible polynomials over a finite field.
\newblock {\em (Russian) Dokl. Akad. Nauk SSSR}, 313(3):552--555, 1990.
\newblock translation in Soviet Math. Dokl. 42 (1991), no. 1, 45--48.

\bibitem{three-coef}
Omran Ahmadi, Faruk Golo\u{g}lu, Robert Granger, Gary McGuire, and Emrah~Sercan
  Y{\i}lmaz.
\newblock Fibre products of supersingular curves and the enumeration of
  irreducible polynomials with prescribed coefficients.
\newblock {\em Finite Fields and Their Applications}, 42(1):128--164, November
  2016.

\bibitem{higher-no-coef}
Robert Granger.
\newblock On the enumeration of irreducible polynomials over gf(q) with
  prescribed coefficients.
\newblock {\em https://arxiv.org/abs/1610.06878}, 2016.

\bibitem{cattell}
K.~Cattell, C.R. Miers, F.~Ruskey, M.~Serra, and J.~Sawada.
\newblock The number of irreducible polynomials over gf(2) with given trace and
  subtrace.
\newblock {\em J. Combin. Math. Combin. Comput.}, 47:31 -- 64, 2003.

\bibitem{yucasmullen}
Joseph~L. Yucas and Gary~L. Mullen.
\newblock Irreducible polynomials over gf(2) with prescribed coefficients.
\newblock {\em Discrete Mathematics}, 274(1--3):265 -- 279, 2004.

\bibitem{lidl}
R.~Lidl and H.~Niederreiter.
\newblock {\em Finite Fields}.
\newblock Number v. 20, pt. 1 in EBL-Schweitzer. Cambridge University Press,
  1997.

\bibitem{cyclotomic}
Lawrence~C. Washington.
\newblock {\em Introduction to Cyclotomic Fields}.
\newblock EBL-Schweitzer. Springer-Verlag New York, 1997.

\bibitem{ireland}
Michael~Rosen Kenneth~Ireland.
\newblock {\em UA Classical Introduction to Modern Number Theory}.
\newblock Springer-Verlag New York, 1990.

\bibitem{kuzmin2}
E.~N. Kuz'min.
\newblock Irreducible polynomials over a finite field and an analogue of gauss
  sums over a field of characteristic 2.
\newblock {\em Siberian Mathematical Journal}, 32(6):982--989, 1991.

\bibitem{magma}
W.~Bosma, J.J. Cannon, and C.~Playoust.
\newblock {The Magma algebra system I: The user language}.
\newblock {\em J. Symb. Comput.}, 24(3/4):235--265, 1997.

\end{thebibliography}
\end{document}